\documentclass[]{interact}

\usepackage{epstopdf}
\usepackage[caption=false]{subfig}

\usepackage[numbers,sort&compress]{natbib}
\bibpunct[, ]{[}{]}{,}{n}{,}{,}
\renewcommand\bibfont{\fontsize{10}{12}\selectfont}
\makeatletter
\def\NAT@def@citea{\def\@citea{\NAT@separator}}
\makeatother

\theoremstyle{plain}
\newtheorem{theorem}{Theorem}[section]
\newtheorem{lemma}[theorem]{Lemma}

\newtheorem{proposition}[theorem]{Proposition}

\theoremstyle{definition}
\newtheorem{definition}[theorem]{Definition}

\newtheorem{assumption}[theorem]{Assumption}

\theoremstyle{remark}
\newtheorem{remark}{Remark}

\usepackage{marvosym}

\usepackage{graphicx}%
\usepackage{multirow}%
\usepackage{amsmath,amssymb,amsfonts}%
\usepackage{amsthm}%
\usepackage{mathrsfs}%
\usepackage[title]{appendix}%
\usepackage{xcolor}%
\usepackage{textcomp}%
\usepackage{manyfoot}%
\usepackage{booktabs}%
\usepackage{algorithm}%
\usepackage{algorithmicx}%
\usepackage{algpseudocode}%
\usepackage{listings}%
\usepackage{multirow}
\usepackage[shortlabels]{enumitem}

\usepackage{hyperref}
\usepackage{tikz}
\usetikzlibrary{arrows,arrows.meta,intersections,shapes.geometric,positioning,calc,patterns}

\DeclareMathOperator*{\argmin}{\arg\min}

\begin{document}


\title{Convergence Conditions for Stochastic Line Search Based Optimization of Over-parametrized Models}

\author{
\name{Matteo Lapucci\textsuperscript{1}, Davide Pucci\textsuperscript{1,\Letter}\thanks{\Letter \, Contact D.\ Pucci. Email: davide.pucci@unifi.it}}
\affil{\textsuperscript{1} Department of Information Engineering, University of Florence, Via di Santa Marta 3, Firenze, 50139, Italy}
}

\maketitle

\begin{abstract}
In this paper, we deal with algorithms to solve the finite-sum problems related to fitting over-parametrized models, that typically satisfy the interpolation condition. In particular, we focus on approaches based on stochastic line searches and employing general search directions. We define conditions on the sequence of search directions that guarantee finite termination and bounds for the backtracking procedure. Moreover, we shed light on the additional property of directions needed to prove fast (linear) convergence of the general class of algorithms when applied to PL functions in the interpolation regime. From the point of view of algorithms design, the proposed analysis identifies safeguarding conditions that could be employed in relevant algorithmic frameworks. In particular, it could be of interest to integrate stochastic line searches within momentum, conjugate gradient or adaptive preconditioning methods.
\end{abstract}


\section{Introduction}

\label{sec:intro}

In this work we are interested in the finite-sum minimization problem
\begin{equation} \label{eq:fin_sum_problem}
	\min_{x \in \mathbb{R}^n} f(x) = \frac{1}{N} \sum_{i=1}^{N} f_i(x)
\end{equation}
where $f_i: \mathbb{R}^n \rightarrow \mathbb{R}$ are differentiable, possibly nonconvex functions for all $i \in \{1, \dots, N\}$ and $N$ is a very large, yet finite number. Supervised learning tasks are, in essence, instances of problem \eqref{eq:fin_sum_problem}; the success and great diffusion of deep-learning applications has thus driven loads of interest to the development of tailored algorithms for this class of problems.  Effective approaches in this setting primarily rely on stochastic gradient descent (SGD) \cite{robbins1951} and its accelerated variants \cite{polyak1964, nesterov1983}; incremental gradient methods indeed take advantage of low per-iteration cost and data redundancy to quickly get close to decently accurate solutions, without giving up asymptotic convergence guarantees (even though in expectation); we refer the reader to the major survey \cite{bottou2018} for a thorough discussion on stochastic gradient methods and their application in data science contexts.

Then, over the past decade, adaptive approaches \cite{duchi2011, tieleman2012, zeiler2012, kingma2014adam} have emerged as consistent boosts for the standard SGD framework. In particular, Adam algorithm \cite{kingma2014adam} is now largely considered the safest choice in highly nonconvex scenarios. At this point, however, a substantial gap between theoretical knowledge and observed behaviors is often underlined; a full understanding about why SGD-type algorithms and Adam in particular work so well in nonconvex settings has long been lacking.

In the meantime, variance-reduced approaches \cite{schmidt2017, rie2013, gower2023} have been proposed as alternative SGD methods achieving faster convergence rates than the base algorithm and in line with standard gradient descent. However, these methods demand a significant overhead of additional computational resources, either memory or exact gradient evaluations, making them impractical for use with complex architectures or very large datasets. In fact, no speed-up seems to be obtainable w.r.t.\ SGD in nonconvex deep learning tasks \cite{defazio2019}.

Finally, in recent years, novel analyses of stochastic gradient methods have been carried out specifically in the interpolation regime, under assumptions that appear to be realistic in deep learning tasks.  Under these hypotheses, convergence rates matching those of ``full-batch'' methods were finally shown to hold for SGD \cite{ma2018, vaswani2020}. 

This novel view also allowed to devise reasoned techniques to select the step size with fast convergence guarantees in an adaptive manner \cite{mutschler2020, loizou2021}. 
A particularly relevant recent development concerns the employment of line searches in the stochastic setting. An Armijo-type line search, extending the classical Armijo rule to the finite-sum setup, was successfully defined \cite{vaswani2019}; more recently, a nonmonotone version of the stochastic line search was proposed \cite{galli2023}. In both cases, the convergence rate of standard gradient descent is retrieved under suitable assumptions.

Stochastic line search approaches under interpolation, however, have mainly been studied - and implemented - assuming an unbiased estimator of the negative gradients is employed as a search direction. This is in contrast with the study of classical nonlinear optimization algorithms, where the pool of possible search directions is vast. To the best of our knowledge, the only attempt to combine stochastic line searches with different directions than the stochastic gradient comes from \cite{fan2023}: here, a suitable adjustment is proposed for the line search condition to work with any direction of descent for the current stochastic objective function; yet, the descent condition is not enough to guarantee convergence for the overall algorithmic scheme.

In this manuscript, we address this issue. Firstly, we define conditions that guarantee not only the finite termination of the line search, but also suitable bounds on the obtained step size and on the number of backtrack steps. Then, the overall linear convergence result under interpolation is proved with an additional assumption, that does not require the search direction to be obtained as an uncorrelated perturbation of the stochastic gradient.

In addition to the theoretical value of the analysis, the results in the paper allow to additionally identify conditions to be used in algorithmic safeguards, so that popular algorithms like momentum-type methods or even Adam can be globalized by means of restart or corrections strategies \cite{powell1977restart,chan2022nonlinear,fan2023,lapucci2024globally}.

The rest of the paper is organized as follows: in Section \ref{sec:prelim} we review the main concepts and state-of-the art methods for the optimization of over-parametrized learning models, with particular emphasis on those based on stochastic line searches. In Section \ref{sec:prelim_discussion} we rigorously formalize the setting for the subsequent analyses. Then, in Section \ref{sec:line_search}, we define the first set of conditions on the directions, which allows to retrieve the main properties of the line search procedure. Finally, the convergence of the overall framework is discussed in Section \ref{sec:full_ana}: linear convergence is proved, under assumptions that are reasonable with over-parametrized learning problems, and only assuming an additional requirement for the search directions which is not restrictive on the structure of the direction itself.
We finally give some concluding remark in Section \ref{sec:conc}.

\section{Preliminaries and Related Works}
\label{sec:prelim}
The finite-sum problem \eqref{eq:fin_sum_problem} is usually handled through \textit{incremental} (or \textit{stochastic}) \textit{algorithms} like stochastic gradient descent. The fundamental idea behind this class of methods is to consider, at each iteration, a cheap approximation of the gradient $\nabla f(x)$; the evaluation of exact derivatives is indeed expensive, as $N$ is assumed to possibly be very large. Bottou et al.\ \cite{bottou2018} offer a detailed presentation of this family of algorithms when an unbiased estimator $g_k(x^k)$ of $\nabla f(x^k)$ is available (i.e., $\mathbb{E}_k[g_k(x^k)] = \nabla f(x^k)$), or when a stochastic Newton or quasi-Newton direction is considered. Here and for the rest of the manuscript, $\mathbb{E}_k[\cdot]$ represents the conditional expectation w.r.t.
$x^k$, i.e., $\mathbb{E}_k[\cdot]=\mathbb{E}[\cdot| x^k]$, whereas $\mathbb{E}[\cdot]$ denotes the total expectation. 
The SGD framework discussed in \cite{bottou2018} covers directions of the form 
\begin{equation*}
	d_k = -H_kg_k(x^k),\qquad g_k(x^k) = \dfrac{1}{|B_k|}\sum\limits_{i \in B_k} \nabla f_i (x^k)
\end{equation*}
where, for all $k \in \mathbb{N}$, $B_k \subseteq \{1, \dots, N\}$ and $H_k$ is a symmetric positive definite matrix conditionally uncorrelated with $g_k$. Direction $d_k$ is employed to update the current solution by the usual update rule
\begin{equation}
	\label{eq:iterative_scheme}
	x^{k+1} = x^k + \alpha_k d_k,
\end{equation}
where $\alpha_k$ is the step size (often referred to as learning rate by the artificial intelligence community). Note that setting $H_k=I$ we recover the usual mini-batch gradient descent, which in turn clearly collapses to base SGD if $|B_k| = 1$. 
When a suitable decreasing sequence of step sizes $\{\alpha_k\}$ is employed (e.g., if $\sum_{k=1}^{\infty} \alpha_k = \infty$ and $\sum_{k=1}^{\infty} \alpha_k^2 < \infty$), it is possible to guarantee that the sequence generated by algorithm \eqref{eq:iterative_scheme} satisfies $\liminf_{k \rightarrow \infty} \mathbb{E}[||\nabla f (x^k)||] = 0$, i.e., there is at least a subsequence of solutions approaching stationarity in expectation. We again refer the reader to \cite{bottou2018} for further details on these results.

Furthermore, in the nonconvex scenario, a worst-case complexity bound of $\mathcal{O}(\frac{1}{\epsilon^4})$ can be proved \cite{ghadimi2013}, which is in fact a tight bound \cite{arjevani2023lower}. Better complexity bounds of $\mathcal{O}(\frac{1}{\epsilon^2})$ and $\mathcal{O}(\frac{1}{\epsilon})$ can be achieved under convexity and strong convexity assumptions, respectively.
These results prove that, in general, SGD methods are not capable of obtaining the rates of convergence typical of full-batch methods. This observation, however, is heavily in contrast with the results observed when SGD is employed to train deep machine learning models.

A recent development in the analysis of SGD is contributing to fill this theoretical gap, based on the observation that modern machine learning models are usually expressive enough to fit any data point in the training set \cite{liang2020,ma2018}. The so-called \textit{interpolation condition} states that, given $x^* \in \arg \min_{x \in \mathbb{R}^n} f(x)$, then $x^* \in \arg \min_{x \in \mathbb{R}^n} f_i(x)$ for all $i \in \{1,\dots, N\}$.  
Moreover, the actual shape of the objective function in deep learning tasks has been investigated \cite{liu2022}: the landscape has been observed not to be convex - even locally - whereas another property appears to hold in large portions of the space: the \textit{Polyak-Lojasiewicz} (PL) condition \cite{polyak1987introduction}. A function $f$ satisfies the PL condition if $\exists \mu > 0$ such that $2 \mu (f(x) - f^*) \leq ||\nabla f(x)||^2$ for all $x \in \mathbb{R}^n$.
Taking into account interpolation, it was finally possible to match gradient descent complexity with SGD methods: $\mathcal{O}(\frac{1}{\epsilon})$ in the convex case and a linear rate either under strong convexity or, most importantly, PL assumptions \cite{vaswani2019, loizou2021, galli2023}.

These findings opened the way to a new line of research that aims at improving SGD performance in both the convex and nonconvex settings. To better discuss this scenario, let us now define $f_k$\footnote{We make here a slight abuse of notation: by $f_k$ we will refer to estimators of $f$, whereas $f_i$ denotes a term in the finite-sum.} and $g_k$ as (cheap) unbiased estimators of $f$ and $\nabla f$, respectively, satisfying
$$\mathbb{E}_k[f_k(x)] = f(x),\qquad \mathbb{E}_k[g_k(x)] = \nabla f(x),\qquad g_k(x) = \nabla f_k(x)$$
for all $x\in\mathbb{R}^n$. These properties of course hold for any pair $(f_k,g_k)$ such that $f_k(x) = \frac{1}{|B_k|}\sum_{i \in {B_k}}f_i(x)$.
Interpolation implies a tighter bond between the progress in $f_{k}(x)$ and the expected progress of the true loss $f(x)$, allowing some SGD algorithmic frameworks to achieve the aforementioned fast rates of convergence.

Loizou et al.\ \cite{loizou2021} proposed an SGD method, with provably fast convergence under interpolation, based on a stochastic adaptation of the Polyak step size \cite{polyak1987} for guessing the most promising step size at each iteration. This approach is however sensitive to the upper bound imposed on stepsizes. 

A different path to achieve fast convergence thanks to a proper selection of the step size builds upon line search techniques. 
Line searches were introduced in the stochastic setting to select the appropriate step size for SGD at each iteration. Vaswani et al. \cite{vaswani2019}, in particular, introduced a suitable Armijo-type condition \cite{armijo1966} for the stochastic case that provably works under interpolation. The peculiarities of this approach are:
\begin{enumerate}[i.]
	\item the progress in the current stochastic objective approximation is measured; 
	\item the sufficient decrease term is based on the stochastic gradient direction.
\end{enumerate}
The cost of checking this condition is reasonable in large scale scenarios, as it only requires an additional evaluation of $f_k$ for each tested step size $\alpha$; a suitable step size $\alpha_k$ can thus be easily computed by a classical backtracking procedure; 
the stochastic line search (SLS) approach can then be shown to provide, in a finite number of backtracks, a step size within a suitable interval related to the Lipschitz constants of $\nabla f_k$. This preliminary result is then combined with interpolation or related assumptions to prove the nice convergence rates for SGD.

A strict decrease condition imposed looking at an approximate objective function might be a restrictive requirement, especially in highly nonlinear and nonconvex scenario. A nonmonotone relaxation to the Armijo-type condition has thus been proposed \cite{galli2023}, with the aim of both reducing the number of mini-batch function evaluations caused by repeated backtracks and accepting aggressive step sizes more frequently. 
The employment of the nonmonotone condition in line search based mini-batch GD is shown not to alter convergence speed. Moreover, exploiting a suitable initial guess step size, related to the stochastic Polyak step, the resulting algorithm, named PoNoS, proves to be computationally very efficient at solving nonconvex learning tasks. 

\medskip
In this work, we discuss theoretical results for optimization algorithms based on stochastic line searches with search directions of a general form. The algorithmic frameworks we are going to analyze are thus characterized by the following two main elements:
\begin{enumerate}[(a)]
	\item the update rule is given by 
	\begin{equation}
		\label{eq:gen_update rule}
		x^{k+1} =  x^k+\alpha_kd_k;
	\end{equation}
	\item the step size $\alpha_k$ is obtained according to
	\begin{equation}
		\label{eq:gen_sls_rule}
		\alpha_k= \max_{j=0,1,\ldots}\{\alpha_0^k\delta^j\mid f_{k}(x^k+\alpha_0^k\delta^j d_k)\le f_{k}(x^k) +\gamma \alpha_0^k\delta^j d_k^T\nabla f_k(x^k) \},
	\end{equation}
	so that it satisfies the general stochastic Armijo condition \cite{fan2023}
	\begin{equation}
		\label{eq:gen_sls_cond}
		f_{k} (x^k + \alpha_k d_k) \leq f_{k} (x^k) + \gamma \alpha_k d_k^T \nabla f_{k} (x^k).
	\end{equation}
\end{enumerate}
In other words, equation \eqref{eq:gen_sls_rule} defines the backtracking procedure, where the initial stepsize $\alpha_0^k$ is decreased by a factor $\delta$ for $j$ times, until the Armijo sufficient decrease condition \eqref{eq:gen_sls_cond} is satisfied. In \cite{fan2023}, it is argued that it is sufficient to somehow impose $d_k^T\nabla f_k(x^k)<0$ to guarantee the well-definiteness - i.e., finite termination - of the line search procedure; while this is certainly true, we will highlight that this condition is not sufficient to obtain convergence results for the whole algorithmic framework, and we will discuss this issue in detail.

\section{The Interpolation Setting}
\label{sec:prelim_discussion}
Before starting the analysis, we need to provide some rigorous definitions of key concepts for the subsequent discussion. First, function $f$ is said to be $L$-smooth if it is continuously differentiable with Lipschitz-continuous gradient, i.e., 
$$\|\nabla f(x)-\nabla f(y)\|\le L\|x-y\|\quad \forall\,x,y\in\mathbb{R}^n;$$
for any $L$-smooth function, the following property (often referred to as the \textit{descent lemma}) holds:
$$f(x) \leq f(y) + \nabla f(y)^T (x - y) + \frac{L}{2} ||x - y||^2,\quad\forall\,x,y\in\mathbb{R}^n.$$

From here onward, we will assume that $f$ is $L$-smooth, that any $f_k$ is $L_k$-smooth and, consequently, that $L\le L_\text{max} = \max_{k}L_k$. We will also assume that $f$ is bounded below and has a global minimizer $x^*$.

We then recall some property often reasonably assumed to hold in the interpolation setting. We start by formally characterizing interpolation itself.

\begin{definition}[{\cite[Def. 3-4]{mishkin2020interpolation}}]
	Let $f:\mathbb{R}^n\to\mathbb{R}$ be an L-smooth finite-sum function. We say that 
	\begin{itemize}
		\item $f$ satisfies the \textit{minimizer interpolation property} if, for all $x^*\in\argmin_x f(x)$  we have $x^*\in\argmin_x f_k(x)$ for any $k$; 
		\item $f$ satisfies the \textit{stationarity interpolation property} if, for all $x^*$ such that $\nabla f(x^*)=0$ we have $\nabla f_k(x^*)=0$ for any $k$. 
	\end{itemize}
\end{definition}
Minimizers interpolation implies stationarity interpolation at all global solutions, so the former property implies the other if $f$ is invex, i.e., if all stationary points are global solutions. The converse is true if all functions $f_k$ are also invex. Otherwise, in the general case, one condition does not imply the other. We next introduce a useful condition for analyzing over-parametrized scenarios.

\begin{definition}[{\cite[Th. 4]{polyak1963gradient}}]
	Let $f:\mathbb{R}^n\to\mathbb{R}$ be an $L$-smooth function and let $f^*$ be its minimum over $\mathbb{R}^n$. Function $f$ satisfies the \textit{Polyak-Lojasiewicz} (PL) condition if there exists $\mu > 0$ such that, for all $x \in \mathbb{R}^n$, we have
	\begin{equation*}
		2 \mu (f(x) - f^*) \leq ||\nabla f(x)||^2.
	\end{equation*}
\end{definition}
The PL inequality basically tells us that the function value does not increase faster than the (squared) size of the gradient as we move away from optimality. Similarly as strong convexity,
the PL property implies that every stationary point is a global solution to the minimization problem. Indeed, strong convexity implies the PL property, which in turn implies invexity.  However, differently than strongly convex functions, PL functions do not necessarily have a unique minimizer.
As already discussed in this manuscript, recent literature shows that variants of the PL condition can be reasonably assumed to study the optimization landscape associated with over-parametrized models \cite{liu2022}.

We then need assumptions to bound the growth rate of the variance of gradients size for individual stochastic functions $f_k$. Two types of conditions have been used in recent literature.
\begin{definition}[{\cite[Sec. 2]{schmidt2013fast}}]
	Function $f:\mathbb{R}^n\to\mathbb{R}$  satisfies the \textit{Strong Growth Condition} (SGC) if there exists $\rho>0$ such that, for any point $x\in\mathbb{R}^n$ and for any $k$,  
	\begin{equation*}
		\mathbb{E}_k[||\nabla f_k(x)||^2] \leq \rho ||\nabla f(x)||^2.
	\end{equation*}
\end{definition}

\begin{definition}[{\cite[Sec. 5]{vaswani2019fast}}]
	Let $f:\mathbb{R}^n\to\mathbb{R}$ be an L-smooth function and let $f^*$ be its minimum over $\mathbb{R}^n$. Function $f$ satisfies the \textit{Weak Growth Condition} (WGC) with constant $\rho$ if, for all $x\in\mathbb{R}^n$ and for any $k$, 
	\begin{equation*}
		\mathbb{E}_k[||\nabla f_k(x)||^2] \leq 2 \rho L [f(x) - f^*].
	\end{equation*} 
\end{definition}

It might be worth remarking that, since $\|\mathbb{E}_k[\nabla f_k(x)]\|^2 = \|\nabla f(x)\|^2\le \mathbb{E}_k[\|\nabla f_k(x)\|^2]$ by Jensen's inequality, actually the constant $\rho$ in the SGC necessarily has to be greater or equal to 1.

We shall now note that the three considered regularity properties are closely tied to each other and with the interpolation conditions. Firstly, it is immediate to observe that the SGC directly implies stationarity interpolation. Actually, the SGC also implies the WGC under $L$-smoothness assumptions \cite{mishkin2020interpolation}, which means the latter condition is indeed weaker than the former one. 

Stationarity interpolation is further satisfied by global minimizers under the WGC, whereas WGC is implied by minimizers interpolation (see \cite[Lemma 4]{galli2023}). Moreover, the following result holds.
\begin{lemma}[{\cite[Proposition 1]{vaswani2019fast}}]
	If $f$ is L-smooth, satisfies the WGC with constant $\rho$ and the PL condition with constant $\mu$, then it satisfies the SGC with constant $\frac{\rho L}{\mu}$. 
\end{lemma}
The above implication in the end tells us that, under PL assumptions, SGC and WGC are equivalent to each other, with minimizers interpolation being a stronger condition and stationarity interpolation a weaker one. These relationships are summarized in Figure \ref{fig:relationships_conditions}.

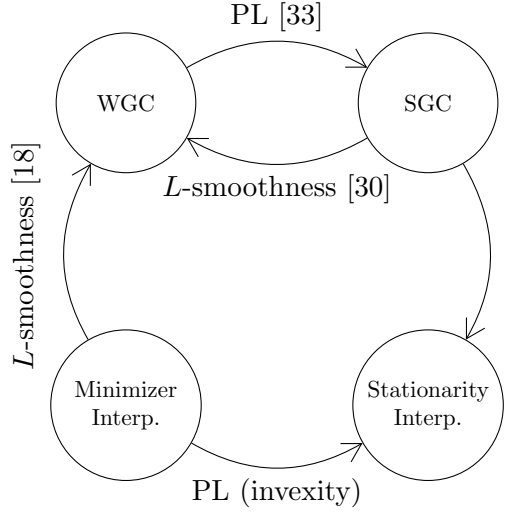
\begin{figure}
	\centering
	\begin{tikzpicture}
		\draw (0,4) node(wgc)[draw,circle, minimum size=0.05,style={scale=0.8}, text width=2cm, align=center]{WGC};
		
		\draw (4,4) node(sgc)[draw,circle, minimum size=0.05,style={scale=0.8}, text width=2cm, align=center]{SGC};
		
		\draw (0,0) node(mi)[draw,circle, minimum size=0.05,style={scale=0.8}, text width=2cm, align=center]{Minimizer Interp.};
		
		\draw (4,0) node(si)[draw,circle, minimum size=0.05,style={scale=0.8}, text width=2cm, align=center]{Stationarity Interp.};
		
		\draw[-{Straight Barb[angle'=60,scale=3]}] (wgc)to[out=30,in=150] node[above]{PL \cite{vaswani2019fast}} (sgc);
		\draw[-{Straight Barb[angle'=60,scale=3]}] (sgc)to[out=-150,in=-30] node[below]{$L$-smoothness \cite{mishkin2020interpolation}} (wgc);
		\draw[-{Straight Barb[angle'=60,scale=3]}] (mi)to[out=120,in=-120] node[above, shift={(-0.5cm, -1.7cm)}]{\rotatebox{90}{$L$-smoothness \cite{galli2023}}} (wgc);
		\draw[{Straight Barb[angle'=60,scale=3]}-] (si)to[out=-150,in=-30] node[below]{PL (invexity)} (mi);
		\draw[-{Straight Barb[angle'=60,scale=3]}] (sgc)to[out=-60,in=+60] node[right]{} (si);
	\end{tikzpicture}
	\caption{Relationships between SGC, WGC and Interpolation under suitable assumptions.}
	\label{fig:relationships_conditions}
\end{figure}

For the analysis of stochastic gradient methods, usually a boundedness assumption regarding the variance of the gradients estimator is also made. In this context, by variance of a random vector $v$ we refer, somewhat improperly, to the quantity
\begin{multline*}
	\text{Var}(v) = \mathbb{E}\left[\left\|v-\mathbb{E}[v]\right\|^2\right] = \mathbb{E}\left[\sum_{i=1}^{n}(v_i-\mathbb{E}[v_i])^2\right] = \sum_{i=1}^n\mathbb{E}[(v_i-\mathbb{E}[v_i])^2]  \\= \sum_{i=1}^n \text{Var}(v_i) = \sum_{i=1}^{n}\left(\mathbb{E}[v_i^2]-\mathbb{E}[v_i]^2\right) = \mathbb{E}\left[\sum_{i=1}^{n}v_i^2\right]-\left\|\mathbb{E}[v]\right\|^2 = \mathbb{E}[\|v\|^2]-\left\|\mathbb{E}[v]\right\|^2.
\end{multline*}
Using a similarly improper terminology, we can refer to the covariance between two random vectors $u$ and $v$ as
\begin{align*}
	\text{Cov}(u,v) &= \mathbb{E}\left[(u-\mathbb{E}[u])^T(v-\mathbb{E}[u])\right] = \mathbb{E}\left[\sum_{i=1}^{n}(u_i-\mathbb{E}[u_i])(v_i-\mathbb{E}[v_i])\right]\\&=\sum_{i=1}^{n}\mathbb{E}[(u_i-\mathbb{E}[u_i])(v_i-\mathbb{E}[v_i])]  = \sum_{i=1}^n \text{Cov}(u_i,v_i) = \sum_{i=1}^n \left(\mathbb{E}[u_iv_i]- \mathbb{E}[u_i]\mathbb{E}[v_i]\right) \\&=\mathbb{E}\left[\sum_{i=1}^{n}u_iv_i\right] -\mathbb{E}[u]^T\mathbb{E}[v] = \mathbb{E}[u^Tv] - \mathbb{E}[u]^T\mathbb{E}[v].
\end{align*}
Among the above equalities, the following one will be of particular relevance in the next sections:
$$\mathbb{E}[u^Tv] = \mathbb{E}[u]^T\mathbb{E}[v]+\text{Cov}(u,v).$$
The covariance between the search directions and the stochastic gradients will be an important quantity for the upcoming analysis.

\section{Conditions for a suitable line search}
\label{sec:line_search}
We are finally able to state properties concerning SGD-type algorithms employing rather general search directions and line searches under interpolation. 
As we have already pointed out earlier in the manuscript, the simple descent condition $d_k^T\nabla f_k(x^k)<0$ would be sufficient to ensure finite termination of the line search. However, similarly as in classical, deterministic nonlinear optimization, that condition is not enough to
\begin{enumerate}[(a)]
	\item guarantee a minimum step size obtained by the line search;
	\item guarantee a lower bound on the number of backtrack steps at each iteration.
\end{enumerate} 

Note that, for the particular set of results that is going to follow, there is no need of any of the conditions on $f$ defined in Section \ref{sec:prelim_discussion} except those concerning $L$-smoothness.
Then, the key assumption on the search direction employed in \eqref{eq:gen_update rule} that we will require to obtain all the forthcoming theoretical results is the following one.

\begin{assumption}
	\label{ass:ass1}
	The sequence of search directions is \textit{stochastic-gradient related} (SGR), i.e., there exist $c_1,c_2>0$ (independent on $k$) such that, for all $k=0,1,\ldots$, the following conditions hold:
	$$\|d_k\|\le c_1\|g_k(x^k)\|,\qquad d_k^Tg_k(x^k)\le -c_2\|g_k(x^k)\|^2.$$
\end{assumption}
\medskip \noindent It is easy to see that if, for example, $d_k=-H_kg_k(x^k)$ and the sequence $\{H_k\}$ satisfies the bounded eigenvalues condition $0< c_2\le \lambda_\text{min}(H_k)\le \lambda_\text{max}(H_k)\le c_1$, then the conditions from Assumption \ref{ass:ass1} hold:
\begin{gather*}
	\|d_k\| = \|H_kg_k(x^k)\|\le \lambda_\text{max}(H_k)\|g_k(x^k)\|\le c_1\|g_k(x^k)\|,\\
	d_k^Tg_k(x^k) = -g_k(x^k)H_kg_k(x^k)\le-\lambda_{\text{min}}(H_k)\|g_k(x^k)\|^2\le  -c_2\|g_k(x^k)\|^2.
\end{gather*}
Note that the above derivation does not require any hypothesis about $H_k$ being conditionally uncorrelated to $g_k(x^k)$. In fact, we shall outline that the line search is performed in a substantially deterministic setting, with a function $f_k$ which is well determined throughout iteration $k$. Thus, we can directly proceed stating the properties related to the Armijo-type line search condition and algorithm. The proofs of these results mostly follow those from classical literature, e.g., \cite{cartis2015worst,cartis2022evaluation}; we fully report them here for the sake of completeness.
\begin{lemma}
	\label{lemma:armijo}
	Let the sequence of directions $\{d_k\}$ satisfy Assumption \ref{ass:ass1}. Let the randomly drawn function $f_k$ considered at any iteration $k$ be an $L_k$-smooth function. Then, the Armijo-type condition \eqref{eq:gen_sls_cond} is satisfied at iteration $k$ for all $\alpha\in[0,\alpha_\text{low}^k]$, where $\alpha_\text{low}^k=\frac{2c_2(1-\gamma)}{c_1^2 L_k}.$
\end{lemma}
\begin{proof}
	Let us assume that a step $\alpha$ does not satisfy the Armijo-type condition, i.e.,
	$$f_k(x^{k}+\alpha d_k)>f_k(x^k)+\gamma \alpha g_k(x^k)^Td_k.$$
	By the $L_k$-smoothness of $f_k$, we also have that
	\begin{equation*}
		f_k(x^{k}+\alpha d_k) \leq f_k(x^k) + \alpha g_k(x^k)^Td_k + \frac{\alpha^2 L_k}{2} ||d_k||^2.
	\end{equation*}
	Combining the two inequalities, we get
	$$\gamma\alpha g_k(x^k)^Td_k<\alpha g_k(x^k)^Td_k + \frac{\alpha^2 L_k}{2} ||d_k||^2,$$
	i.e.,
	$$\alpha(1-\gamma ) g_k(x^k)^Td_k + \frac{\alpha^2 L_k}{2} ||d_k||^2>0.$$
	Bounding the quantities $g_k(x^k)^Td_k$ and $\|d_k\|$ according to Assumption \ref{ass:ass1}, we furthermore get
	$$-c_2\alpha(1-\gamma ) \|g_k(x^k)\|^2 + \frac{c_1^2\alpha^2 L_k}{2} ||g_k(x^k)||^2>0.$$
	Dividing by $\alpha \|g_k(x^k)\|^2$ and rearranging, we finally get
	$$\alpha>\frac{2c_2(1-\gamma)}{c_1^2L_k}=\alpha_\text{low}^k,$$
	which completes the proof.
\end{proof}

\begin{proposition}
	\label{prop:armijo}
	Let the sequence of directions $\{d_k\}$ satisfy Assumption \ref{ass:ass1} and the randomly drawn functions $f_k$ considered at any iteration $k$ be $L_k$-smooth functions, and let $L_\text{max} = \max_{k}L_k$. Further assume that $\alpha_0^k\le \alpha_\text{max}$ for all $k$. 
	Then, at each iteration $k$, the number of backtrack steps $j_k$ (and thus
	of additional stochastic evaluations of $f$) is bounded above by 
	$$j_k\le j^*=\max\left\{0,\left\lceil\log_{1/\delta}\frac{\alpha_\text{max}}{\alpha_\text{low}}\right\rceil\right\},$$
	where $\alpha_\text{low} = \frac{2c_2(1-\gamma)}{c_1^2L_\text{max}}$ and the step size $\alpha_k$ is bounded by
	$$\alpha_k\ge \min\{\alpha_0^k,\delta\alpha_\text{low}\} .$$
\end{proposition}
\begin{proof}
	By the definitions of $\alpha_\text{low}$ and $L_\text{max}$, we certainly have for all $k$ that $\alpha_\text{low}\le \alpha_\text{low}^k$ for all $k$.
	Now, let $j^*$ be the smallest integer such that $$\delta^{j^{*}}\alpha_\text{max}\le \alpha_\text{low}.$$
	For any $k$ we thus have
	$$\delta^{j^{*}}\alpha_0^k \le \delta^{j^{*}}\alpha_\text{max}\le \alpha_\text{low}\le \alpha_\text{low}^k.$$
	By Lemma \ref{lemma:armijo}, the step $\delta^{j^{*}}\alpha_0^k$ is thus guaranteed to satisfy the Armijo-type condition. Therefore, by the Armijo line search, we have $j_k\le j^*$.
	By the definition of $j^*$ we also have that
	$$\delta^{j^{*}}\le \frac{\alpha_\text{low}}{\alpha_\text{max}},$$
	i.e.,
	$$j^*\ge \log_{\delta}\frac{\alpha_\text{low}}{\alpha_\text{max}} = \log_{1/\delta}\frac{\alpha_\text{max}}{\alpha_\text{low}}.$$ 
	Actually, $j^*$ is the smallest positive integer such that the above inequality holds, so 
	$$j^*=\max\left\{0,\left\lceil\log_{1/\delta}\frac{\alpha_\text{max}}{\alpha_\text{low}}\right\rceil\right\}.$$
	
	Now, if $\alpha_0^k<\alpha_\text{low}$, there is certainly no need of backtracking, $j_k=0$ and $\alpha_k=\alpha_0^k$.
	Otherwise, if $\alpha_0^k>\alpha_\text{low}$, we have $$j_k\le j^*=\left\lceil\log_{1/\delta}\frac{\alpha_\text{max}}{\alpha_\text{low}}\right\rceil\le \log_{1/\delta}\frac{\alpha_\text{max}}{\alpha_\text{low}}+1,$$
	and then 
	$$\frac{1}{\delta^{j_k-1}}\le \frac{\alpha_\text{max}}{\alpha_\text{low}},\qquad\text{i.e.,}\qquad \delta^{j_k-1}\ge \frac{\alpha_\text{low}}{\alpha_\text{max}}\ge \frac{\alpha_\text{low}}{\alpha_0^k},$$
	and finally
	$\delta\alpha_\text{low}\le \delta^{j_k}\alpha_0^k=\alpha_k.$
	The proof is thus complete.
\end{proof}

The above result is important for two main reasons: first, it guarantees a lower bound on the step size to be used at each iteration, enhancing the substantial progress of the overall algorithm, regardless of the particular stochastic function $f_k$ drawn at iteration $k$; second, the upper bound on the number of backtracks allows convergence rate and complexity results in terms of number of iterations to directly translate into equivalent results in terms of number of stochastic function and gradient evaluations.

\begin{remark}
	We expect the results in Lemma \ref{lemma:armijo} and Proposition \ref{prop:armijo} to be be extendable to the case of a nonmonotone line search, exploiting the reasonings from \cite{cartis2015worst} and \cite{galli2023}. We prefer to focus on the monotone case in our analysis for the sake of simplicity, as the focus of this manuscript is on directions, rather than on the particular type of line search.
\end{remark}

\section{Conditions on Directions for Global Convergence}
\label{sec:full_ana}
We finally turn to the convergence properties of the generalized class of SGD-type algorithms. Of course, for the analysis of the overall algorithmic framework we need to assume that both the PL property and minimizers interpolation (and thus SGC and WGC) hold.
Moreover, to give the main convergence result, we need to ensure that the sequence $\{d_k\}$ satisfies a further property in addition to Assumption \ref{ass:ass1}. 

In particular, we need bounds concerning the expected values (conditioned to $x^k$) of $d_k$ and $g_k(x^k)$, which are usually assumed also in the analysis of SGD methods in absence of interpolation \cite[Sec.\ 4.1]{bottou2018}:
\begin{equation}
	\label{eq:bottou-like}
	\|\mathbb{E}_k[d_k]\|\le C_1\|\nabla f(x^k)\|,\qquad \mathbb{E}_k[d_k]^T\nabla f(x^k)\le -C_2\|\nabla f(x^k)\|^2,
\end{equation}
for some $C_1,C_2>0$.

When $d_k$ is obtained according to $d_k=-H_kg_k(x^k)$ with $\{H_k\}$ satisfying the bounded eigenvalues condition \textit{and $H_k$ being conditionally uncorrelated with $g_k(x^k)$}, the conditions immediately hold:
\begin{multline*}
	\|\mathbb{E}_k[d_k]\| = \|\mathbb{E}_k[H_kg_k(x^k)]\| = \|\mathbb{E}_k[H_k]\mathbb{E}_k[g_k(x^k)]\| = \|\mathbb{E}_k[H_k]\nabla f(x^k)\| =\\ \|\mathbb{E}_k[H_k\nabla f(x^k)]\|\le \mathbb{E}_k[\|H_k\nabla f(x^k)\|]\le \mathbb{E}_k[\lambda_\text{max}(H_k)\|\nabla f(x^k)\|]\le c_1\|\nabla f(x^k)\|,
\end{multline*}
and
\begin{align*}
	\mathbb{E}_k[d_k]^T\nabla f(x^k) &= \mathbb{E}_k[-H_kg_k(x^k)]^T\nabla f(x^k) = -\mathbb{E}_k[g_k(x^k)]^T\mathbb{E}_k[H_k]\nabla f(x^k) \\&= -\nabla f(x^k)\mathbb{E}_k[H_k]\nabla f(x^k) = -\mathbb{E}_k[\nabla f(x^k)H_k\nabla f(x^k)]\\&\le -\mathbb{E}_k[\lambda_\text{min}(H_k)\|\nabla f(x^k)\|^2] = -c_2\|\nabla f(x^k)\|^2.
\end{align*}
Stochastic Newton and Quasi-Newton approaches (like, e.g., those from \cite{meng2020fast}) thus enjoy this type of property if applied with uniform positive definiteness safeguards: the use of line searches is therefore sound within these methods. On the other hand, we cannot make an analogous straight reasoning for
\begin{itemize}
	\item adaptive SGD methods, like Adam, where the definition of the preconditioner $H_k$ depends on the values of $g_k(x^k)$;
	\item methods with momentum-type terms in the search direction, where $$H_k=I-\beta\,\text{diag}\left(\frac{x_1^k-x_1^{k-1}}{(g_k(x^k))_1},\ldots,\frac{x_n^k-x_n^{k-1}}{(g_k(x^k))_n}\right);$$
	\item conjugate gradient type directions, where the direction has a similar structure as momentum methods and the value of $\beta$ is also dependent on $g_k(x^k)$.
\end{itemize}

We are thus interested in carrying out the analysis with a general characterization of the search directions, not tied to uncorrelated preconditioning operations on the stochastic gradient. 
To this aim, we state an assumption related to the covariance of the search directions and the stochastic gradients.
\begin{assumption}
	\label{ass:ass2}
	There exists $c_3 > 0$ such that the sequence of search directions $\{d_k\}$ satisfies the following property:
	$${\text{Cov}_k}(d_k,g_k(x^k))\ge - c_3 {\text{Var}_k}(g_k(x^k)).$$
\end{assumption}
Assumption \ref{ass:ass2} somehow asks for the employed search direction not to vary in contrast and infinitely more than the stochastic gradient does. The assumption holds, for example, if the variance of $d_k$ is bounded by the variance of $g_k(x^k)$: 
\begin{equation}\label{eq:bound_var}
		{\text{Var}_k}(d_k(x^k))\le {c}_3 {\text{Var}_k}(g_k(x^k)).
\end{equation} 
We indeed have
	\begin{align*}
		|{\text{Cov}_k}(d_k,g_k(x^k))| &= \left|\sum_{i=1}^{n}\text{Cov}_k((d_k)_i,(g_k(x^k))_i)\right| \le  \sum_{i=1}^{n}\left|\text{Cov}_k((d_k)_i,(g_k(x^k))_i)\right|\\ 
		& \le \sum_{i=1}^{n}\sqrt{\text{Var}_k((d_k)_i)}\sqrt{\text{Var}_k((g_k(x^k))_i)}\\
		& = \left(
		\sqrt{\text{Var}_k((d_k)_1)}\; \dots\; \sqrt{\text{Var}_k((d_k)_n)}
			\right) \begin{pmatrix}
				\sqrt{\text{Var}_k((g_k(x^k))_1)} \\
				\vdots \\
				\sqrt{\text{Var}_k((g_k(x^k))_n)}
			\end{pmatrix}\\
		& \le \sqrt{\text{Var}_k(d_k)}\sqrt{\text{Var}_k(g_k(x^k))}  \le {c}_3 \text{Var}_k(g_k(x^k)),
	\end{align*}
	where the second inequality comes from Pearson's formula and the third one from Cauchy-Schwartz inequality $u^Tv \leq \|u\|\|v\|$ with $u_i = \sqrt{\text{Var}_k((d_k)_i)}$ and $v_i=\sqrt{\text{Var}_k((g_k(x^k))_i)}$.

In other words, Assumption \ref{ass:ass2} is guaranteed to hold if the variance of the direction does not grow more than linearly with the variance of the stochastic gradients. 
This assumption seems thus reasonable to make, especially in conjunction with the common hypothesis that $\text{Var}_k(g_k(x^k))\le M$. In the following remark we specifically discuss the interesting case of momentum-type directions. 

\begin{remark}
	The pure momentum-type direction with constant $\beta$ satisfies \eqref{eq:bound_var} and thus Assumption 5.1 with $c_3=1$:
	$$\text{Var}_k(-g_k(x^k)+\beta s_k) = \text{Var}_k(-g_k(x^k)) +\text{Var}_k(\beta s_k) = \text{Var}_k(-g_k(x^k)),$$
	where the first equality follows from the independence of $g_k(x^k)$ on both $s_k$ and $\beta$ and the second one follows from $\beta$ being a constant and $s_k$ being deterministic at iteration $k$. 
	Thus, if the momentum direction passes the stochastic-gradient-related test at each iteration, we are guaranteed that it satisfies both Assumptions 4.1 and 5.1. 
	
	Of high interest in practice is also the case of momentum with restarts based on the stochastic-gradient related conditions: formally, we have $d_k = -g_k(x^k) + \beta_k s_k$, where $$\beta_k= \begin{cases}
		\bar{\beta}&\text{if }\|p_k\|\le c_1\|g_k(x^k)\|\text{ and }p_k^Tg_k(x^k)\le -c_2\|g_k(x^k)\|^2,\\
		0&\text{otherwise},
	\end{cases}$$
	being $p_k = -g_k(x^k) + \bar{\beta} s_k$.
	We argue that Assumption 5.1 in this case is very reasonable, as it holds if we assume that $\|s_k\|\le M$ and  $\text{Var}_k(\beta_k) \le \hat{c}_3 \text{Var}_k(g_k(x^k))$.
	The former condition can be immediately enforced in practice by a simple clipping operation that is independent of $g_k(x^k)$. The latter assumption will instead be commented later.
	The variance of $d_k$ can now be rewritten as follows:
	\begin{align*}
		\text{Var}_k(-g_k(x^k)+\beta_k s_k) = \text{Var}_k(g_k(x^k)) + \|s_k\|^2\text{Var}_k(\beta_k) - 2\text{Cov}_k(g_k(x^k),\beta_k s_k)
	\end{align*}
	where we used the fact that $s_k$ is constant given $x^k$; we can furthermore write
	\begin{align*}
		|\text{Cov}_k(g_k(x^k),\beta_k s_k)| &\le \sum_{i=1}^{n} \left|\text{Cov}_k((g_k(x^k))_i,\beta_k (s_k)_i)\right|\\&\le 
		\sum_{i=1}^{n} \sqrt{\text{Var}_k((g_k(x^k))_i)}\sqrt{\text{Var}_k(\beta_k (s_k)_i)}\\& = \sqrt{\text{Var}_k(\beta_k)}\sum_{i=1}^{n} \sqrt{\text{Var}_k((g_k(x^k))_i)} |(s_k)_i|\\&\le \sqrt{\text{Var}_k(\beta_k)}\sqrt{\text{Var}_k(g_k(x^k))}\|s_k\|
	\end{align*}
	where we used that $\beta_k$ is constant for all $i$ and we exploited Cauchy-Schwartz inequality on the dot product $u^Tv$ where $v_i = |(s_k)_i|$ and $u_i=\sqrt{\text{Var}_k((g_k(x^k))_i)}$.
	We can then proceed by bounding \begin{align*}
		\text{Var}_k(-g_k(x^k)+\beta_k s_k) &\le \text{Var}_k(g_k(x^k)) + \|s_k\|^2\text{Var}_k(\beta_k) + 2\sqrt{\text{Var}_k(\beta_k)}\sqrt{\text{Var}_k(g_k(x^k))}\|s_k\|\\&\le  \text{Var}_k(g_k(x^k))\left(1+M^2\hat{c}_3+2M\sqrt{\hat{c}_3}\right),
	\end{align*}
	which again concludes the proof for Assumption 5.1 being satisfied.
	
	\smallskip
	
	Concerning the assumption $\text{Var}_k(\beta_k) \le \hat{c}_3 \text{Var}_k(g_k(x^k))$: $\beta_k$ follows a Bernoulli distribution, with a probability $p_k$ of taking the value $\bar{\beta}$, and its variability only depends on the variability of $g_k(x^k$). Without committing to a formal derivation, heuristically the assumption appears reasonable to make, especially if we note that it provably holds in the extreme cases:
	\begin{itemize}
		\item if $g_k(x^k)$ is deterministic, then we have that the variance of both $\beta_k$ and $g_k(x^k)$ is zero and the condition is satisfied;
		\item if the variance of the stochastic gradients gets very large, we will certainly have $\text{Var}_k(\beta_k)\le \hat{c}_3\text{Var}_k(g_k)$, as for Bernoulli distribution it is known that $\text{Var}_k(\beta_k)=p(1-p)\bar{\beta}^2\le \frac{\bar{\beta}^2}{4}$.
	\end{itemize}
\end{remark}

\medskip
Combined with Assumption \ref{ass:ass1} and the SGC, the condition from Assumption \ref{ass:ass2} actually implies for $\{d_k\}$ the stronger assumptions of the form \eqref{eq:bottou-like}.
\begin{lemma}
	\label{lemma:bottou}
	Let the SGC condition and Assumptions \ref{ass:ass1}-\ref{ass:ass2} hold, with $c_2>c_3(1-\frac{1}{\rho})$. Then, we have:
	\begin{gather}
		\label{eq:bottou_conditions_a}
		\|\mathbb{E}_k[d_k]\|\le c_1\sqrt{\rho}\|\nabla f(x^k)\|,\\\label{eq:bottou_conditions_b}\mathbb{E}_k[d_k]^T\nabla f(x^k)\le -\bigg(c_2-c_3\bigg(1-\frac{1}{\rho}\bigg)\bigg)\|\nabla f(x^k)\|^2.
	\end{gather}
\end{lemma}
\begin{proof}
	From Jensen's inequality, the first condition from Assumption \ref{ass:ass1} and the SGC, we immediately get that
	$$\|\mathbb{E}_k[d_k]\|^2 \leq \mathbb{E}_k[\|d_k\|^2] \leq c_1^2 \mathbb{E}_k[\|g_k(x^k)\|^2] \leq \rho c_1^2 \|\nabla f (x^k)\|^2.$$
	Recalling that $\rho, c_1 > 0$, we obtain
	$$\|\mathbb{E}_k[d_k]\| \leq c_1 \sqrt{\rho} \|\nabla f (x^k)\|,$$
	i.e., condition \eqref{eq:bottou_conditions_a}.
	
	\medskip
	
	\noindent Given $g_k(x^k)$ is an unbiased estimator of $\nabla f(x^k)$ and by the definition of the covariance, we have that
	$$\mathbb{E}_k[d_k]^T\nabla f(x^k) = \mathbb{E}_k[d_k]^T \mathbb{E}_k[g_k(x^k)] = \mathbb{E}_k[d_k^T g_k(x^k)] - {\text{Cov}_k} (d_k,g_k(x^k)),$$
	from which, using Assumption \ref{ass:ass2}, we get that
	$$
	\begin{aligned}
		\mathbb{E}_k[d_k]^T\nabla f(x^k) \leq & \mathbb{E}_k[d_k^T g_k(x^k)] + c_3 {\text{Var}_k}(g_k(x^k)) \\
		= & \mathbb{E}_k[d_k^T g_k(x^k)] + c_3 \mathbb{E}_k[\|g_k(x^k)\|^2] - c_3 \|\mathbb{E}_k[g_k(x^k)]\|^2.
	\end{aligned}
	$$
	Applying the second condition from Assumption \ref{ass:ass1} to the quantity $\mathbb{E}_k[d_k^T g_k(x^k)]$ leads to 
	$$\mathbb{E}_k[d_k]^T\nabla f(x^k) \leq - c_2 \mathbb{E}_k[\|g_k(x^k)\|^2] + c_3 \mathbb{E}_k[\|g_k(x^k)\|^2] - c_3 \|\mathbb{E}_k[g_k(x^k)]\|^2.$$
	We can now apply the SGC to the rightmost term to obtain
	$$\mathbb{E}_k[d_k]^T\nabla f(x^k) \leq - c_2 \mathbb{E}_k[\|g_k(x^k)\|^2] + c_3 \mathbb{E}_k[\|g_k(x^k)\|^2] - \frac{c_3}{\rho} \mathbb{E}_k[\|g_k(x^k)\|^2],$$
	and thus 
	\begin{align*}
		\mathbb{E}_k[d_k]^T\nabla f(x^k) &\le - \bigg(c_2-c_3\bigg(1-\frac{1}{\rho}\bigg)\bigg)\mathbb{E}_k[\|g_k(x^k)\|^2]\\&\le - \bigg(c_2-c_3\bigg(1-\frac{1}{\rho}\bigg)\bigg)\|\mathbb{E}_k[g_k(x^k)]\|^2,
	\end{align*}
	where the last step, that completes the proof, comes from Jensen's inequality ($- \|\mathbb{E}_k[g_k(x^k)]\|^2 \geq - \mathbb{E}_k[\|g_k(x^k)\|^2]$).
\end{proof}

\begin{remark}
	\label{remark:ass2}
	We shall note that, to guarantee condition \eqref{eq:bottou_conditions_a}, Assumption \ref{ass:ass1} coupled with the SGC is sufficient. The additional Assumption \ref{ass:ass2} is thus needed to guarantee the second condition in \eqref{eq:bottou-like}.
\end{remark}

The above result is a key from the algorithmic standpoint. Conditions of the form \eqref{eq:bottou-like} involve the true gradient and the conditional expected value of the search direction, and are thus not checkable as we only have access to a realization of the random variables. On the other hand, the conditions from Assumption \ref{ass:ass1} concern the realization itself of the random variables: for given $c_1$ and $c_2$, it is possible to numerically assess whether the conditions are verified or not. 

Safeguard clauses based on these conditions can thus be employed in algorithms; if, at a certain iteration, one of the two inequalities happens to be false for the obtained search direction, we can for example adopt a restart strategy \cite{powell1977restart,chan2022nonlinear,fan2023} and switch to the simple SGD direction, which is guaranteed to possess all the suitable properties needed for convergence; of course, more sophisticated strategies could also be devised (see, e.g., \cite{lapucci2024globally} for the momentum case).

We are finally able to state the main convergence result of the paper.

\begin{theorem}
	\label{th:main}
	Assume $f:\mathbb{R}^n\to\mathbb{R}$ is an $L$-smooth function, satisfying the PL condition and the minimizer interpolation property, and that the randomly drawn functions $f_k$ considered at any iteration $k$ are $L_k$-smooth functions, with $L\le L_\text{max} = \max_{k}L_k$. Let the sequence $\{d_k\}$ satisfy Assumptions \ref{ass:ass1}-\ref{ass:ass2} with $c_2>c_3(1-\frac{1}{\rho})$. Then, if $\{x^k\}$ is the sequence produced by iterating updates of the form \eqref{eq:gen_update rule}, where $\alpha_k$ is selected according to \eqref{eq:gen_sls_rule} with $\delta\alpha_{\text{low}}\le\alpha_0^k\le \alpha_{\text{max}}$ for all $k$, then the following property holds:
	\begin{equation}
		\label{eq:main_thesis}
		\mathbb{E}[f(x^{k+1}) - f(x^*)] \leq (\eta\alpha_{\text{max}})^k(f(x^0)-f(x^*)),
	\end{equation}
	where $\eta= \left(\frac{L_{\text{max}} c_1^2}{2 c_2}\left(\frac{1}{\gamma}+\frac{1}{\delta(1-\gamma)}\right) - 2 \bigg(c_2-c_3\bigg(1-\frac{1}{\rho}\bigg)\bigg) \mu   \right)$.
	
	\noindent Therefore, if all the constants involved are such that
	$0<\eta<\frac{1}{\alpha_\text{max}}$, the rate of convergence is linear with an $\mathcal{O}(\log(\frac{1}{\epsilon}))$ iteration, stochastic function evaluations and stochastic gradient evaluations complexity to achieve an $\epsilon$-accurate solution in expectation.
\end{theorem}
\begin{proof}
	By the $L$-smoothness of $f$ we have
	\begin{align*}
		f(x^{k+1}) - f(x^k) &\leq \nabla f(x^k)^T (x^{k+1} - x^k) + \frac{L}{2} \|x^{k+1} - x^k\|^2 \\
		&= \alpha_k \nabla f(x^k)^T d_k + \frac{\alpha_k^2 L}{2} \|d_k\|^2.
	\end{align*}
	Dividing by $\alpha_k > 0$ and applying both conditions from Assumption \ref{ass:ass1} we get
	\begin{align*}
		\frac{f(x^{k+1}) - f(x^k)}{\alpha_k} &\leq \nabla f(x^k)^T d_k + \frac{\alpha_k L}{2} \|d_k\|^2 \\
		& \leq \nabla f(x^k)^T d_k + \frac{\alpha_k L c_1^2}{2} \|g_k (x^k)\|^2 \\
		& \leq \nabla f(x^k)^T d_k - \frac{\alpha_k L c_1^2}{2 c_2} g_k (x^k)^T d_k.
	\end{align*}
	From the stochastic Armijo condition \eqref{eq:gen_sls_cond} we have
	$$-\alpha_k g_k (x^k)^T d_k \leq \frac{f_k(x^k) - f_k (x^{k+1})}{\gamma},$$
	that we can use to obtain
	$$\frac{f(x^{k+1}) - f(x^k)}{\alpha_k}\leq \nabla f(x^k)^T d_k + \frac{L c_1^2}{2 \gamma c_2} (f_k(x^k) - f_k (x^{k+1})).$$
	Let $x^*$ be a global minimizer of $f$. By interpolation we have that $f_k(x^{k+1}) \ge f_k(x^*)$. Therefore,
	$$\frac{f(x^{k+1}) - f(x^k)}{\alpha_k}\leq \nabla f(x^k)^T d_k + \frac{L c_1^2}{2 \gamma c_2} (f_k(x^k) - f_k (x^*)).$$
	By taking the conditional expectation w.r.t.\ $k$ and recalling that $f_k$ is an unbiased estimator of $f$, we have 
	$$\mathbb{E}_k \left[ \frac{f(x^{k+1}) - f(x^k)}{\alpha_k} \right] \leq \nabla f(x^k)^T \mathbb{E}_k[d_k] + \frac{L c_1^2}{2 \gamma c_2} (f (x^k) - f (x^*)).$$
	By Lemma \ref{lemma:bottou}, we can now use \eqref{eq:bottou_conditions_b}, letting $\sigma= (c_2-c_3(1-\frac{1}{\rho}))$, and the PL condition to obtain
	\begin{align*}
		\mathbb{E}_k \left[ \frac{f(x^{k+1}) - f(x^k)}{\alpha_k} \right] & \leq - \sigma \|\nabla f(x^k)\|^2 + \frac{L c_1^2}{2 \gamma c_2} (f(x^k) - f(x^*)) 
		\\
		& \leq - 2 \sigma \mu (f (x^k) - f (x^*))  + \frac{L c_1^2}{2 \gamma c_2} (f(x^k) - f (x^*))
		\\
		& = \left( \frac{L c_1^2}{2 \gamma c_2} - 2 \sigma \mu  \right)(f (x^k) - f (x^*)).
	\end{align*}
	Subtracting on both sides the quantity $\mathbb{E}_k [f(x^*) / \alpha_k]$ and rearranging, we obtain
	$$
	\mathbb{E}_k \left[ \frac{f(x^{k+1}) - f(x^*)}{\alpha_k} \right]
	\leq \mathbb{E}_k \left[ \frac{f(x^k) - f(x^*)}{\alpha_k} \right] + \left( \frac{L c_1^2}{2 \gamma c_2} - 2 \sigma \mu  \right)(f (x^k) - f (x^*)).
	$$
	From Proposition \ref{prop:armijo} we know that $\alpha_k\ge \min\{\alpha_0^k,\delta\alpha_\text{low}\} = \alpha_\text{min}$, therefore we have
	\begin{align*}
		\mathbb{E}_k \left[ \frac{f(x^{k+1}) - f(x^*)}{\alpha_k} \right]
		& \leq \frac{(f(x^k) - f(x^*))}{\mathbb{E}_k [\alpha_k]} + \left(\frac{L c_1^2}{2 \gamma c_2} - 2 \sigma \mu  \right)(f (x^k) - f (x^*)) \\
		& \leq  \left(\frac{L c_1^2}{2 \gamma c_2} - 2 \sigma \mu  + \frac{1}{\alpha_\text{min}} \right)(f (x^k) - f (x^*)).
	\end{align*}
	Recalling that $\alpha_k \leq \alpha_{\text{max}}$, that $L \leq L_{\text{max}}$ and taking the total expectation we obtain
	\begin{equation*}
		\mathbb{E} \left[ f(x^{k+1}) - f(x^*) \right]
		\leq  \alpha_{\text{max}} \left(\frac{L_{\text{max}} c_1^2}{2 \gamma c_2} - 2 \sigma \mu  + \frac{1}{\alpha_\text{min}} \right)\mathbb{E}[f (x^k) - f (x^*)].
	\end{equation*}
	We now focus on the quantity $\left(\frac{L_{\text{max}} c_1^2}{2 \gamma c_2} - 2 \sigma \mu  + \frac{1}{\alpha_\text{min}} \right)$.
	Rearranging some terms, and recalling that, by assumption, $\alpha_0^k \ge \delta \alpha_{\text{low}}$ we get
	\begin{multline*}
		\left(\frac{L_{\text{max}} c_1^2}{2 \gamma c_2} - 2 \sigma \mu  + \frac{1}{\alpha_\text{min}} \right)= \left(\frac{L_{\text{max}} c_1^2}{2 \gamma c_2} - 2 \bigg(c_2-c_3\bigg(1-\frac{1}{\rho}\bigg)\bigg) \mu  + \frac{c_1^2 L_{\text{max}}}{2 c_2 \delta (1- \gamma)} \right)\\=
		\left(\frac{L_{\text{max}} c_1^2}{2 c_2}\left(\frac{1}{\gamma}+\frac{1}{\delta(1-\gamma)}\right) - 2 \bigg(c_2-c_3\bigg(1-\frac{1}{\rho}\bigg)\bigg) \mu   \right) = \eta
	\end{multline*}
	Thus, we have
	$$\mathbb{E} \left[ f(x^{k+1}) - f(x^*) \right]
	\leq  \eta\alpha_{\text{max}} \mathbb{E}[f (x^k) - f (x^*)]$$
	Now, we can recursively apply the above inequality from 0 to $k$ to finally obtain \eqref{eq:main_thesis}.
\end{proof}

Of course, many of the constants appearing in the result from Theorem \ref{th:main} will not be known in practice; yet, manipulating the values of $\delta$ and $\gamma$, that are hyperparameters of line search algorithms, it would always be possible to obtain $\eta>0$; then, we also know that, for some suitable values of $\alpha_{\text{max}}$, the linear convergence rate is actually achievable. Interestingly, if we fall back to the simple case $d_k=-g_k(x^k)$, we can substitute $c_1=c_2=c_3=1$, getting $\eta=(\frac{L_{\text{max}}}{2}(\frac{1}{\gamma}+\frac{1}{\delta(1-\gamma)}) - \frac{2 \mu}{\rho} )$ and recovering an analogous result to the one from \cite{galli2023}.

\section{Conclusions}
\label{sec:conc}
In this work, we analyzed an algorithmic framework based on stochastic line searches and general search directions to tackle finite-sum optimization of over-parametrized models. The algorithm is shown to possess a linear convergence rate when applied to PL functions satisfying the interpolation condition, as long as the search directions are guaranteed to satisfy suitable assumptions.
 
In particular, we proved that the key condition for convergence in this setting is for the directions to be related to the current stochastic gradient. From the computational perspective, this property is  much more convenient to handle than other conditions often considered in the literature on SGD-type methods. Indeed, it can be practically checked and thus employed to define safeguarding and restart strategies. The result from this manuscript thus opens the way to a sound integration of stochastic line search and popular momentum-type directions and preconditioning strategies.

In future research, further investigations about the assumption on the covariance between the stochastic gradient and the search direction might be of very significant impact. On the other hand, the implementation and testing of restarting techniques based on the stochastic gradient related conditions would surely be of interest.

\section*{Acknowledgement(s)}
The authors are grateful to the editors and referees involved in the peer review process of this manuscript: their constructive comments helped us improve the overall quality of the paper.

\section*{Disclosure statement}

The authors have no competing interests to declare that are relevant to the content of this article.

\section*{Funding}

No funding was received for conducting this study.

\section*{Data Availability Statement}
Data sharing is not applicable to this article as no new data were created or analyzed in this study.

%

%

%

\bibliographystyle{tfnlm}
\bibliography{interactnlmsample}


\end{document}